\documentclass[12pt]{article}
\usepackage{amsthm}
\usepackage{amssymb}
\usepackage{mathtools,xparse}
\usepackage{latexsym}
\usepackage{enumerate}
\usepackage[title]{appendix}
\newtheorem{theorem}{Theorem}[section]
\newtheorem{proposition}[theorem]{Proposition}

\newtheorem{lemma}[theorem]{Lemma}

\theoremstyle{definition}
\newtheorem{definition}[theorem]{Definition} % definition numbers are dependent on theorem numbers
\newtheorem{remark}[theorem]{Remark}

\title{Rigidity for the isoperimetric inequality of negative effective dimension on weighted Riemannian manifolds}
\date{\today}
\author{MAI, Cong Hung\footnote{Department of Mathematics, Kyoto University, Kyoto 606-8502, Japan (hongmai@math.kyoto-u.ac.jp)}}

\begin{document}
	
	\maketitle
	
	\begin{abstract}
	We study, on a weighted Riemannian manifold of Ric$_{N} \geq K > 0$ for $N < -1$, when equality holds in the isoperimetric inequality. Our main theorem asserts that such a manifold is necessarily isometric to the warped product $\mathbb{R} \times_{\cosh(\sqrt{K/(1-N)}t)} \Sigma^{n-1}$ of hyperbolic nature, where $\Sigma^{n-1}$ is an $(n-1)$-dimensional manifold with lower weighted Ricci curvature bound and $\mathbb{R}$ is equipped with a hyperbolic cosine measure. This is a similar phenomenon to the equality condition of Poincar\'e inequality. Moreover, every isoperimetric minimizer set is isometric to a half-space in an appropriate sense. 
	\end{abstract}
	
	%\tableofcontents
	
	\section{Introduction}%%%%%%%%%%%%%%%%%%
	%%%%%%%%%%%
	
	The isoperimetric inequality is a classical topic in comparison geometry with the history tracing back to the ancient Greece. Most of the work on isoperimetric problem has been done in Euclidean spaces and Riemannian manifolds. Recently, the isoperimetric problem can be formulated in greater generality in weighted manifolds, meaning Riemannian manifolds equipped with arbitrary (smooth, positive) measures (See \cite{Mor}).
	
	In a weighted manifold, the Ricci curvature is modified into the \emph{weighted Ricci curvature} $\mathrm{Ric}_N$ involving a parameter $N$ which is sometimes called the effective dimension. Recently, the developments on the \emph{curvature-dimension condition} in the sense of Lott, Sturm and Villani have shed new light on the theory of curvature bounds for weighted manifolds. The curvature-dimension condition $\mathrm{CD}(K,N)$ is a synthetic notion of lower Ricci curvature bounds for metric measure spaces. The parameters $K$ and $N$ are usually regarded as ``a lower bound of the Ricci curvature'' and ``an upper bound of the dimension'', respectively. The roles of $K$ and $N$ are better understood when we consider a weighted Riemannian manifold $(M,g,m)$: $\mathrm{CD}(K,N)$ is equivalent to $\mathrm{Ric}_N \ge K$.
	
	Geometric analysis for the weighted Ricci curvature $\mathrm{Ric}_N$
	(also called the \emph{Bakry--\'Emery--Ricci curvature}) including the isoperimetric inequality
	has been intensively studied by Bakry and his collaborators
	in the framework of the \emph{$\Gamma$-calculus} (see \cite{BGL} and \cite{BL}). Recently it turned out that there is a rich theory also for $N \in (-\infty,1]$, though this range seems strange due to the above interpretation of $N$ as an upper dimension bound. For examples, various Poincar\'e-type inequalities (\cite{KM}), the curvature-dimension condition (\cite{Oneg,Oneedle}), the splitting theorem (\cite{WY}) were studied for $N<0$ or $N \le 1$. In our previous paper \cite{Mai}, we studied the rigidity of the Poincar\'e inequality (spectral gap) under the condition Ric$_{N} \geq K > 0$ with $N < -1$, and showed that the sharp spectral gap is achieved only if the space is isometric to the warped product $\mathbb{R} \times_{\cosh(\sqrt{K/(1-N)}t)} \Sigma^{n-1}$ of hyperbolic nature. In this paper, we continue this study to the rigidity problem of the isoperimetric inequality.
	
	The isoperimetric inequality on weighted manifolds satisfying Ric$_{N} \geq K$ and diam$(M) \leq D$ was studied in \cite{Mil1} and \cite{Mil2}. The isoperimetric inequality could also be verified in a gentle way called the \emph{needle decomposition} on Riemannian manifolds developed by Klartag in \cite{Kl}. The idea is to reduce a high dimensional inequality into its one dimensional version on geodesics, which is much easier to verify. Cavalletti and Mondino generalized this method to metric measure spaces satisfying CD$(K,N)$ condition for $N \in (1,\infty)$, and also established the rigidity result for the isoprimetic inequality in \cite{CM}. In this paper, we use the needle decomposition method to consider the rigidity of the isoperimetric inequality under the condition $\mathrm{Ric}_{N} \geq K > 0$ and $N < -1$. The splitting phenomenon in the result implies that the manifold is necessarily isometric to the warped product $\mathbb{R} \times_{\cosh(\sqrt{K/(1-N)}t)} \Sigma^{n-1}$, where $\Sigma^{n-1}$ is an $(n-1)$-dimensional manifold with Ric$_{N-1} \ge K(2-N)/(1-N)$ and $\mathbb{R}$ is equipped with a hyperbolic cosine measure (see Theorem \ref{th:ndimneg}). This is directly related to the lower bound problem of the first nonzero eigenvalue in \cite{Mai} (see Theorem \ref{th:mai}), because the hyperbolic sine of the Lipschitz function used to construct the needle decomposition (called a guiding function in \cite{Kl}) turns out an eigenfunction associated with the smallest eigenvalue of the Laplacian. 
	
	The organization of this article is as follows: In Section 2 we give a brief introduction about weighted Riemannian manifolds, including the weighted Ricci curvature. We also review the basics of the isoperimetric inequality as well as the needle decomposition method. In Section 3 we consider the rigidity problem for the Bakry--Ledoux isoperimetric inequality under the condition $\mathrm{Ric}_\infty \geq K$. In this case we have the isoperimetric splitting with the Gaussian space as shown in \cite{Mor}. We give an alternative proof based on Klartag's needle decomposition, which will be helpful to understand the case of negative effective dimension. Section 4 contains the proof of our main theorem on the rigidity of isoperimetric inequality of negative effective dimension on weighted Riemannian manifolds. The case of negative effective dimension requires some additional technical arguments since the rigidity does not turn out an isometric splitting as in the case $N = \infty$, and the expected eigenfunction is not just a guiding function but the hyperbolic sine of it.
	\medskip
	
	\textit{Acknowledgements}.
	I would like to thank my supervisor, Professor Shin-ichi Ohta,
	for the kind guidance, encouragement and advice he has provided
	throughout my time working on this paper. I also would like to thank Professor Frank Morgan and Professor Emanuel Milman for giving valuable comments on the reference part
	of the first draft of this paper.
	
	\section{Preliminaries}%%%%%%%%%%%%%%%%%
	%%%%%%%%%%%
	
	\subsection{Weighted Riemannian manifolds}%%%%%%%%%%%
	%%%%%%%%%%%
	
	A weighted Riemannian manifold $(M,g,m)$ will be a pair of
	a complete, connected, boundaryless manifold $M$
	equipped with a Riemannian metric $g$ and a measure $m = e^{-\psi}\mathrm{vol}_{g}$,
	where $\psi \in C^{\infty}(M)$ and $\mathrm{vol}_{g}$ is the standard volume measure on $(M,g)$.
	On $(M,g,m)$, we define the weighted Ricci curvature as follows:
	
	\begin{definition}[Weighted Ricci curvature]\label{df:wRic}
		Given a unit vector $v \in U_{x}M$ and $N \in (-\infty,0) \cup [n,\infty]$,
		the \emph{weighted Ricci curvature} $\mathrm{Ric}_N(v)$ is defined by
		\begin{enumerate}[(1)]
			\item $\mathrm{Ric}_{N}(v) :=\mathrm{Ric}_g(v) +\mathrm{Hess}\,\psi(v,v)
			-\displaystyle\frac{\langle \nabla \psi(x),v\rangle^2}{N-n}$ for $N \in (-\infty,0) \cup (n,\infty)$;
			
			\item $\mathrm{Ric}_{n}(v) :=\mathrm{Ric}_g(v) +\mathrm{Hess}\,\psi(v,v)$
			if $\langle \nabla\psi(x),v\rangle = 0$, and $\mathrm{Ric}_n(v):=-\infty$ otherwise;				
			
			\item $\mathrm{Ric}_{\infty}(v) :=\mathrm{Ric}_g(v) +\mathrm{Hess}\,\psi(v,v)$,
		\end{enumerate}
		where $n=\dim M$ and $\mathrm{Ric}_g$ denotes the Ricci curvature of $(M,g)$.
		The parameter $N$ is sometimes called the \emph{effective dimension}.
		We also define $\mathrm{Ric}_{N}(cv) :=c^{2}\mathrm{Ric}_{N}(v)$ for $c \ge 0$.
	\end{definition}
	
	Note that if $\psi$ is constant then the weighted Ricci curvature coincides with
	$\mathrm{Ric}_g(v)$ for all $N$.
	When $\mathrm{Ric}_{N}(v) \geq K$ holds for some $K\in \mathbb{R}$ and all unit vectors $v \in TM$,
	we will write $\mathrm{Ric}_{N}\geq K$.
	By definition,
	\[ \mathrm{Ric}_n(v) \le \mathrm{Ric}_N(v) \le \mathrm{Ric}_{\infty}(v) \le \mathrm{Ric}_{N'}(v) \]
	holds for $n \le N<\infty$ and $-\infty<N'<0$,
	and $\mathrm{Ric}_N(v)$ is non-decreasing in $N$ in the ranges $(-\infty,0)$ and $[n,\infty]$.
	
	Note that the curvature-dimension condition $\mathrm{CD}(K,N)$ in the sense of Lott--Sturm--Villani is equivalent to $\mathrm{Ric}_N \ge K$
	(see \cite{vRS, StI,StII,LV} and \cite{Oint} as well for the Finsler analogue).
	
	We also define the weighted Laplacian with respect to $m$.
	
	\begin{definition}[Weighted Laplacian]\label{df:Lap}
		The \emph{weighted Laplacian} (also called the \emph{Witten Laplacian})
		of $u \in C^\infty(M)$ is defined as follows:
		\[ \Delta_{m}u := \Delta u - \langle \nabla u,\nabla\psi \rangle. \]
	\end{definition}
	
	Notice that the Green formula (the integration by parts formula)
	\[ \int _{M} u\Delta_{m}v \,dm =-\int _{M}\langle \nabla u,\nabla v\rangle \,dm =\int _{M} v\Delta_{m} u\,dm \]
	holds provided $u$ or $v$ belongs to $C_c^{\infty}(M)$
	(smooth functions with compact supports) or $H^1_0(M)$.
	
	An important result on weighted manifolds satisfying $\mathrm{Ric}_N \geq K > 0$ is the lower bound of the first nonzero eigenvalue of the weighted Laplacian (equivalently, Poincar\'e inequality) $\lambda_{1} \geq KN/(N-1)$. The case of equality was studied in \cite[Theorem 2]{CZ} ($N = \infty$) and \cite[Theorem 4.5]{Mai} ($N < -1$) as a Ric$_{N}$-counterpart to the classical Obata theorem in \cite{Ob}:
	 
	 \begin{theorem}(Cheng-Zhou)\label{th:poincare}
	{
	Let $(M,g,m)$ be a weighted Riemannian manifold satisfying $\emph{Ric}_{\infty} \geq K$ for $K > 0$. 
	
	The equality $\lambda_{1} = K$ with an eigenfunction $u$ associated with $\lambda_{1}$ implies that $M$ is isometric to the product space $\Sigma^{n-1} \times \mathbb{R}$ as weighted Riemannian manifolds, where $\Sigma^{n-1}  = u^{-1}(0)$ is an $(n-1)$-dimensional manifold with $\mathrm{Ric}_{\infty} \ge K$ and $\mathbb{R}$ is equipped with the Gaussian measure $e^{-Kt^2/2}dt$. Moreover $u(x,t)$ (as the function on the product space) is constant on $\Sigma^{n-1} \times \{t\}$, can be chosen as $u(x,t) = t$.
	}
	\end{theorem}
	
	\begin{theorem} (Mai) \label{th:mai}
		Let $(M,g,m)$ be a complete weighted Riemannian manifold
		satisfying $\mathrm{Ric}_{N} \ge K$ for some $N<-1$ and $K>0$,
		and $m(M)<\infty$.
		The equality $\lambda_{1} = KN/(N-1)$ with an eigenfunction $u$ associated with $\lambda_{1}$ implies that $M$ is isometric to the warped product
			\[ \mathbb{R} \times_{\cosh(\sqrt{K/(1-N)}t)} \Sigma
			=\bigg( \mathbb{R} \times \Sigma,
			 dt^2 +\cosh^2\bigg( \sqrt{\frac{K}{1-N}}t \bigg) \cdot g_{\Sigma} \bigg) \]
			and the measure $m$ is written through the isometry as
			\[ m(dtdx) =\cosh^{N-1}\bigg( \sqrt{\frac{K}{1-N}}t \bigg) \,dt \,m_{\Sigma}(dx), \]
			where $\Sigma^{n-1}  = u^{-1}(0)$ is an $(n-1)$-dimensional weighted Riemannian manifold
			satisfying $\mathrm{Ric}_{N-1} \ge K(2-N)/(1-N)$. 
		\end{theorem}
		
	We remark that the inequality $\lambda \geq KN/(N-1)$ is known to be sharp only for $N \leq -1$, and when $N = -1$ the lower bound is never achieved. It was shown in \cite{KM} that the constant $KN/(N-1)$ is not sharp at least for $N < 0$ close to $0$.
	
	\subsection{Isoperimetric inequalities}
	
	Another important result on weighted Riemannian manifolds with lower Ricci curvature bound is the isoperimetric inequality. To state the isoperimetric inequality, we define \emph{Minkowski's exterior boundary} as follows:

	\begin{equation}{
	m^{+}(A) := \liminf_{\epsilon \downarrow 0} \frac{m(A^{\epsilon})-m(A)}{\epsilon}
	}
    \end{equation}
	for a Borel set $A$, where $A^{\epsilon}$ denotes the $\epsilon$-neighborhood of $A$. Assuming $m(M) < \infty$ (this is the case when Ric$_{\infty} \geq K > 0$), we normalize $m$ as $m(M) = 1$ since such a normalization does not change $\textrm{Ric}_{N}$.
	The \emph{isoperimetric profile} is defined as follows:
	 \[ I_{(M,g,m)}(\theta) := \inf \{m^{+}(A)|A\subset M, \textrm{Borel set with } m(A) = \theta\} \textrm{ for } \theta \in (0,1). \]
	Isoperimetric inequalities under $\mathrm{Ric}_{N} \geq K > 0$ were shown by Le\'vy--Gromov \cite{Gr1,Gr2} ($N = n$), Bayle \cite{Ba} ($N \in (n,\infty$)), Bakry--Ledoux \cite{BL} ($N = \infty$) and Milman \cite{Mil2} ($N < 0$). In fact, Milman \cite{Mil1,Mil2} intensively studied the setting of $\mathrm{Ric}_{N} \geq K$ and $\mathrm{diam}(M) \leq D$. 
	 
	 \begin{theorem} (Isoperimetric inequality) \label{th:iso}
	{
	Let $(M,g,m)$ be a weighted Riemannian manifold satisfying $m(M) = 1$, $\mathrm{diam}(M) \leq D$ with $D\in (0,\infty]$ and $\emph{Ric}_{N} \geq K$ for $N \in (-\infty,0)\cup[n,\infty]$. Then for all $\theta\in(0,1)$, we have 
	$I_{(M,g,m)} (\theta) \geq I_{(K,N,D)}(\theta)$, where $I_{(K,N,D)}$ depends only on $K,N$ and $D$.
	}
	\end{theorem}
	
	 For the precise formula of the function $I_{(K,N,D)}$, we refer to \cite{Mil1,Mil2}. The estimation is sharp in all the parameters $K,N,D$ and the dimension $n$ of the manifold.
	In this paper we consider only the case where $K > 0$ and $N \in (-\infty,0)\cup\{\infty\}$. Then $I_{(K,N,D)}$ is defined as follows:
	
	\[ I_{(K,\infty,\infty) }(\theta):=\sqrt{\frac{K}{2\pi}}e^{\frac{-Ka^{2}(\theta)}{2}} \]
	where $a(\theta) \in \mathbb{R}$ is defined by $ \theta =\int_{-\infty}^{a(\theta)} \sqrt{\frac{K}{2\pi}}e^{\frac{-Ks^{2}}{2}}ds$ and for $D\in (0,\infty)$
	
	\[ I_{(K,\infty,D)}(\theta) := \inf_{\xi \in [-D,0]} f_{\xi,D} (\theta)  \textrm{ with } f_{\xi,D}(\theta) = \frac{e^{-\frac{Kb(\theta)^2}{2}}}{\int_{\xi}^{\xi+D}e^{-\frac{Ks^2}{2}}ds} \] where $b(\theta) \in (\xi,\xi+D)$ is defined by $\theta = \frac{\int_{\xi}^{b(\theta)}e^{-\frac{Ks^2}{2}}ds}{\int_{\xi}^{\xi+D}e^{-\frac{Ks^2}{2}}ds}$. For $N < 0$, we put $\sigma := K/(1-N)$ and define
	
	\[ I_{(K,N,\infty)}(\theta) := \frac{\cosh^{N-1}( {\sqrt{\sigma}}c(\theta))}{\int_{-\infty}^{\infty} \cosh^{N-1}( {\sqrt{\sigma}}s)ds}\]
	where $c(\theta)$ is defined by $\theta =\frac{\int_{-\infty}^{c(\theta)}\cosh^{N-1}( {\sqrt{\sigma}}s)ds}{\int_{-\infty}^{\infty} \cosh^{N-1}( {\sqrt{\sigma}}s)ds}$ and 
	
	\[ I_{(K,N,D)}(\theta) := \min \{K_{1,D}(\theta), K_{2,D}(\theta), K_{3,D}(\theta)\} \] 
	where $K_{1,D}(\theta) := \inf_{\xi\in\mathbb{R}}\frac{\cosh^{N-1}( {\sqrt{\sigma}}d_{1,\xi}(\theta))}{\int_{\xi}^{\xi + D} \cosh^{N-1}( {\sqrt{\sigma}}s)ds}$, $K_{2,D}(\theta) := \inf_{\xi > 0}\frac{\sinh^{N-1}( {\sqrt{\sigma}}d_{2,\xi}(\theta))}{\int_{\xi}^{\xi + D} \sinh^{N-1}( {\sqrt{\sigma}}s)ds}$, $K_{3,D}(\theta) := \frac{e^{(N-1)\sqrt{\sigma}d_{3}(\theta)}}{\int_{0}^{D} e^{(N-1)\sqrt{\sigma}s}ds}$ and $d_{1,\xi}(\theta)$, $d_{2,\xi}(\theta)$, $d_{3}(\theta)$ are defined by 
	\[ \theta = \frac{\int_{\xi}^{d_{1,\xi}(\theta)}\cosh^{N-1}( {\sqrt{\sigma}}s)ds}{\int_{\xi}^{\xi+D} \cosh^{N-1}( {\sqrt{\sigma}}s)ds} = \frac{\int_{\xi}^{d_{2,\xi}(\theta)}\sinh^{N-1}( {\sqrt{\sigma}}s)ds}{\int_{\xi}^{\xi + D} \sinh^{N-1}( {\sqrt{\sigma}}s)ds} = \frac{\int_{0}^{d_{3}(\theta)} e^{(N-1)\sqrt{\sigma}s}ds}{\int_{0}^{D} e^{(N-1)\sqrt{\sigma}s}ds}.\]
	
	We use the following lemma (Proposition 2.1 in \cite{Bob}) to study the equality case in the 1-dimensional isoperimetric inequality of $N = \infty$.
	
	\begin{lemma} (Bobkov) \label{lem:logconc}
	{
	Let $m$ be a log-concave measure on $\mathbb{R}$, then the minimum of $m^{+}(A)$ on the class of all Borel sets $A \subset \mathbb{R}$ with $m(A) = \theta$ coincides with the minimum on the subclass of the intervals $(-\infty,a]$ or $[b,\infty)$.
	}
	\end{lemma}

	\subsection{Needle decomposition in Riemannian geometry}%%%%%%%%%
	%%%%%%%%%%%
	
	This part is mostly taken from \cite{Kl}.
	
	Firstly we define transport rays associated to a Lipschitz function.
	
	\begin{definition}(Transport ray)
	{
	Let $u$ be a 1-Lipschitz function on $M$. We say that $I \subset M$ is a \emph{transport ray} associated with $u$ if $|u(x)-u(y)| = d(x,y)$ for all $x,y \in I$ and for all $J 	\supsetneq I$, there exist $x,y \in J$ with $|u(x) - u(y)| \ne d(x,y)$.
	}
	\end{definition}
	
	The needle decomposition (also called the localization), the main tool of this paper, was stated in one of the main theorems of \cite{Kl}:
	
	\begin{theorem}(Needle decomposition)\label{th:Ndl}
	{
	Let $(M,g,m)$ be a weighted Riemannian manifold satisfying $\emph{Ric}_{N} \geq K$ and $f$ is an integrable function on $M$ with $\int_{M} fdm = 0$ and $\int_{M}|f(x)|d(x_{0},x)m(dx) < \infty$ for some $x_{0} \in M$.
	Then there exists a $\mathrm{1}$-Lipschitz function $u$, a partition $Q$ on $M$, a measure $\nu$ on $Q$ and a family of probability measures $\{\mu_{I}\}_{I \in Q}$ on $M$ such that:
	\begin{enumerate}[{\rm (i)}]
			\item For any measurable set $A$ in $M$, we have
			$m(A) = \int_{Q}\mu_{I} (A) d\nu(I)$. For $\nu$-almost all $I \in Q$, $\mathrm{supp}(\mu_{I}) \subset I$.
			\item For $\nu$-almost all $I \in Q, I$ is a minimizing geodesic and transport ray associated with $u$. Moreover, if $I$ is not a singleton then the weighted Ricci curvature of $(I,|\cdot|,\mu_{I})$ satisfies $\mathrm{Ric}_{N}^{I} \geq K$.
			\item For $\nu$-almost every $I \in Q$, $\int_{I} fd\mu_{I} = 0$.
		\end{enumerate}
	}
	\end{theorem}
    
    We recall Klartag's proof in \cite{Kl} of the isoperimetric inequality by the needle decomposition, reducing the isoperimetic inequality to the 1-dimensional analysis. See also \cite{CM,Oneedle} for the cases of metric measure spaces and Finslerian manifolds, respectively.
	
	Let $(M,g,m)$ satisfy $\mathrm{Ric}_{N} \geq K  > 0$ for $N \in (-\infty,0)\cup\{\infty\}$ with $\mathrm{dim} M \geq 2$, and assume $m(M) = 1$. We will see that $I_{(M,g,m)}(\theta) \geq I_{(K,N,\infty)}(\theta)$ holds for all $\theta \in (0,1)$.
	
	Take a Borel set $A\subset M$ with $m(A) = \theta$. Put $f(x) := 1_{A}(x) - \theta$. Then $\int_{M} fdm = 0$. We obtain a needle decomposition $Q,\nu, \{\mu_{I}\}_{I \in Q}$ associated with $f$ as in Theorem \ref{th:Ndl}. Note that (iii) in Theorem \ref{th:Ndl} implies $\mu_{I}(A) = m(A) = \theta$ for $\nu$-a.e. $I \in Q$. By (i) in Theorem \ref{th:Ndl} and the definition of the boundary measure $m^{+}$, we have
	\begin{equation}\label{equ:orthogonal}
	m^{+}(A) \geq \int_{Q}\mu_{I}^{+} (A) d\nu(I).
	\end{equation}
	By the weighted Ricci curvature bound (ii) in Theorem \ref{th:Ndl}, applying the $1$-dimensional isoperimetric inequality (see the proofs of Lemmas \ref{lem:1dim} and \ref{lem:1dim-neg}) on each needle $I$ yields
	\begin{equation}
	m^{+}(A) \geq \int_{Q} I_{(K,N,\infty)}(\theta) d\nu(I) = I_{(K,N,\infty)}(\theta).
	\end{equation}
	
	\section{Rigidity for Bakry--Ledoux isoperimetric inequality}%%%%%%%%%
	%%%%%%%%%%%
	
	Using needle decomposition argument, we will discuss the equality case of isoperimetric problem on weighted manifolds. For $N \in (1,\infty)$ Cavalletti--Mondino \cite{CM} showed that the equality in the isoperimetric inequality implies that the space is necessarily isometric to the spherical suspension. Their proof relies on the  maximal diameter theorem in \cite{Ke} and is not generalized to $N = \infty$ nor $N < 0$. In this section we consider the case of $N = \infty$.
	The rigidity of this dimension free version of isoperimetric inequality of weighted manifolds was studied by Frank Morgan in \cite[Theorem 18.7]{Mor} by using classical techniques in geometric measure theory. There was also a alternative proof by Raphael Bouyrie using $\Gamma$-calculus in \cite{Bou}. Here we give another proof using needle decomposition.
	
	Firstly, we prove $I_{(K,\infty,D)} > I_{(K,\infty,\infty)}$.
	
	\begin{lemma}\label{lem:1dim1}
	{
	For every $\theta \in (0,1)$, we have 
	\begin{equation}\label{ine:infty}
	  I_{(K,\infty,D)}(\theta) > I_{(K,\infty,\infty)}(\theta).
	\end{equation}
	}
	\end{lemma}
	\begin{proof}
	{
	We abbreviate in this proof $I := I_{(K,\infty,\infty)}$. By the definitions of $a(\theta)$ and $b(\theta)$ in $\S$2.2, we have
	$1 = I(\theta)a'(\theta) = f_{\xi,D}(\theta)b'(\theta)$ for a fixed $\xi \in [-D,0]$. By a straight forward calculation:
	\begin{equation}
	   I'(\theta) = -KI(\theta)a(\theta)a'(\theta) = -Ka(\theta), 
	\end{equation}
	\begin{equation}
	   f_{\xi,D}'(\theta) = -Kf_{\xi,D}(\theta)b(\theta)b'(\theta) = -Kb(\theta). 
	\end{equation}
	Putting $h=f_{\xi,D} - I$, we have $h \geq \min\{h(0),h(1), h(\theta)$ where $I'(\theta) = f'_{\xi,D}(\theta)\}$.
	When $I'(\theta) = f'_{\xi,D}(\theta)$, we have $a(\theta) = b(\theta)$ and 
	
	\[ h(\theta) = \bigg(\frac{1}{\int_{\xi}^{\xi + D}\varphi(s)ds} - \sqrt{\frac{K}{2\pi}}\bigg)\varphi(a(\theta)) \]
	where $\varphi(t) = e^{\frac{-Kt^2}{2}}$. Note also that $h(0) = \frac{\varphi(\xi)}{\int_{\xi}^{\xi + D}\varphi(s)ds}$ and $h(1) = \frac{\varphi(\xi + D)}{\int_{\xi}^{\xi + D}\varphi(s)ds}$. Therefore
	
	\[h(\theta) \geq \varphi (D) \bigg( \frac{1}{\int_{\xi}^{\xi + D}\varphi(s)ds} - \sqrt{\frac{K}{2\pi}} \bigg) > \varphi (D) \bigg( \frac{1}{\int_{-D}^{D}\varphi(s)ds}-\sqrt{\frac{K}{2\pi}} \bigg) > 0. \]
	Taking the infimum over $\xi\in[-D,0]$ shows \eqref{ine:infty}.
	
	}
	\end{proof}
	
	Now we consider the equality case of isoperimetric inequality in 1-dimensional manifolds.
	
	\begin{lemma}($1$-dimensional case)\label{lem:1dim}
	{
	Let $(M,g,m)$ be a weighted Riemannian manifold of dimension $1$ satisfying $\emph{Ric}_{\infty} \geq K > 0$ and suppose $m(M) = 1$. Assume that there exists $\theta_{0} \in (0,1)$ with $I_{(M,g,m)}(\theta_{0}) = I_{(K,\infty,\infty)}(\theta_{0})$. Then $(M,g,m) = (\mathbb{R}, |\cdot|, \sqrt{\frac{K}{2\pi}}e^{-\frac{Kx^{2}}{2}}dx)$ as weighted manifolds. Moreover, if a Borel set $A$ satisfies $m^{+}(A) = I_{(K,\infty,\infty)}(m(A))$ then $A$ is $(-\infty,a]$ or $[b,\infty)$ up to a difference of an $m$-negligible set.
	}
	\end{lemma}
	
	\begin{proof}
	{
	Note first that, since $I_{(K,\infty,D)}(\theta_{0}) > I_{(K,\infty,\infty)}(\theta_{0})$ for $D < \infty$, $M$ is noncompact and hence isometric to $\mathbb{R}$. In order to show that $m$ is Gaussian, we recall the proof of $I_{(M,g,m)} \geq I_{(K, \infty,\infty)}$ in \cite[Theorem 1.2]{Mil1}.
	
	Since $M$ is 1-dimensional and $\textrm{Ric}_{\infty} \geq K$, $\psi$ is a $K$-convex function (we do not assume the smoothness of $\psi$ for later use on each needle). Hence by Lemma \ref{lem:logconc}, it suffices to consider $A$ of the form $(-\infty,a]$ or $[b,\infty)$ as a minimizer of the isoperimetric problem where $m(A) = \theta$ and $m^{+}(A) = I_{(M,g,m)}(\theta)$. We assume $A=(-\infty,a]$, the case $A = [b,\infty)$ is similar.
	
	Since  $\psi$ is $K$-convex, we have
	\begin{equation} \label{equ:plus}
	{\psi(x+t) \geq \psi(x) + \psi'_{+}(x)t + \frac{Kt^2}{2}} \qquad (t>0),
	\end{equation}
	and
	\begin{equation} \label{equ:minus}
	{\psi(x+t) \geq \psi(x) + \psi'_{-}(x)t + \frac{Kt^2}{2}} \qquad  (t<0).
	\end{equation}
	Therefore, for $t>0$,
	\begin{equation} \label{equ:key-epo}
	{e^{-\psi(x+t)} \leq e^{-\psi(x)}e^{-\psi'_{+}(x)t - \frac{Kt^2}{2}}}.
	\end{equation}
	}We obtain the following estimation
	\begin{equation}\label{equ:key}
	m(A^{r}) - m(A) = \int_{A^{r}}dm - \int_{A}dm \leq m^{+}(A) \int_{0}^{r} e^{-\psi'_{+}(a)t - \frac{Kt^2}{2}} dt.
	\end{equation}
	Letting $r$ go to $\infty$, we have
	\begin{equation}\label{equ:inf}
	1 - \theta \leq m^{+}(A) \int_{0}^{\infty} e^{-\psi'_{+}(a)t - \frac{Kt^2}{2}} dt.
	\end{equation}
	Using a similar argument for $M\setminus A$, we also obtain
	
	\begin{equation}\label{equ:minusinf}
	\theta \leq m^{+}(A) \int_{-\infty}^{0} e^{-\psi'_{-}(a)t - \frac{Kt^2}{2}} dt \leq m^{+}(A) \int_{-\infty}^{0} e^{-\psi'_{+}(a)t - \frac{Kt^2}{2}} dt
	\end{equation}
	since $\psi'_{-}(a) \leq \psi'_{+}(a)$ by the $K$-convexity.
	Hence
	\begin{equation}\label{equ:mminf}
	I_{(M,g,m)} (\theta)\geq \inf_{H\in \mathbb{R}} \max\bigg\{\frac{1-\theta}{\int_{0}^{\infty} J_{H}(t) dt}, \frac{\theta}{\int_{-\infty}^{0} J_{H}(t)dt}\bigg\},
	\end{equation}
	where $J_{H}(t) = e^{Ht - \frac{Kt^2}{2}}$.
	When $H$ varies from $-\infty$ to $\infty$, the first (second) term in the right hand side of (\ref{equ:mminf}) varies monotonically from $\infty$ to $0$ ($0$ to $\infty$). Therefore the infimum over $H\in \mathbb{R}$ is attained at the unique point $H_{\theta}$ such that both terms coincide:
	
	\begin{equation}\label{infimum}
	\frac{\int_{0}^{\infty} J_{H_{\theta}}(t) dt}{1-\theta} = \frac{\int_{-\infty}^{0} J_{H_{\theta}}(t)dt}{\theta} = \int_{-\infty}^{\infty} J_{H_{\theta}}(t)dt.
	\end{equation}
	We have by $H_{\theta}t - \frac{Kt^2}{2} = -\frac{K}{2}(t-\frac{H_{\theta}}{K})^2 + \frac{H_{\theta}^2}{2K}$:
	
	\begin{equation}
	\theta = \frac{\int_{-\infty}^{0} J_{H_{\theta}}(t)dt}{\int_{-\infty}^{\infty} J_{H_{\theta}}(t)dt} = \frac{\int_{-\infty}^{\frac{-H_{\theta}}{K}} e^{\frac{-Ks^2}{2}}ds}{\int_{-\infty}^{\infty} e^{\frac{-Ks^2}{2}}ds} = \sqrt{\frac{K}{2\pi}} \int_{-\infty}^{\frac{-H_{\theta}}{K}} e^{\frac{-Ks^{2}}{2}}ds.
	\end{equation}
	Thus $a(\theta) = -H_{\theta}/K$ and
	
	\begin{equation} \label{equ:main}
	I_{(M,g,m)} (\theta)\geq \bigg(\int_{-\infty}^{\infty} J_{H_{\theta}}(t)dt\bigg)^{-1} = \sqrt{\frac{K}{2\pi}}e^{\frac{-Ka^{2}(\theta)}{2}}.
	\end{equation}
	
	Now, if equality holds in \eqref{equ:main} at $\theta_{0}$, then we have equality both in \eqref{equ:plus} and \eqref{equ:minus}, as well as $\psi'_{-} (a) = \psi'_{+} (a)$. Hence $\psi'' \equiv K$ and we obtain $(M,g,m) = (\mathbb{R}, |\cdot|, \sqrt{\frac{K}{2\pi}}e^{-\frac{Kx^{2}}{2}}dx)$ as metric measure spaces. Moreover, by Proposition 3.1 in \cite{CFMP}, if a Borel set $A$ satisfies $m^{+}(A) = I_{(K,\infty,\infty)}(m(A))$ then $A = (-\infty,a]$ or $[b,\infty)$.
	\end{proof}

	Combining Lemma \ref{lem:1dim}, Theorem \ref{th:iso} and Theorem \ref{th:poincare}, we obtain:
	\begin{theorem}(High dimensional case)\label{th:ndim}
	{
	Let $(M,g,m)$ be a weighted Riemannian manifold of dimension $n \geq 2$ satisfying $\emph{Ric}_{\infty} \geq K > 0$ and suppose $m(M) = 1$. Assume that there exists $\theta_{0} \in (0,1)$ with $I_{(M,g,m)}(\theta_{0}) = I_{(K,\infty,\infty)}(\theta_{0})$. Then the following hold:
	\begin{enumerate}[{\rm (i)}]
	        \item $M$ is isometric to the product space $\Sigma^{n-1} \times \mathbb{R}$
		    as weighted Riemannian manifolds,
		    where $\Sigma^{n-1}$ is an $(n-1)$-dimensional manifold with $\emph{Ric}_{\infty} \ge K$
		    and $\mathbb{R}$ is equipped with the Gaussian measure $e^{-\frac{Kx^2}{2}}dx$.
			\item If a Borel set $A$ satisfies $m^{+}(A) = I_{(K,\infty,\infty)}(m(A))$ then $A$ is a half-space of $M$ (with respect to the product structure, the precise meaning will be given in the proof). 
		\end{enumerate}
		
	}
	\end{theorem}
	
	\begin{proof}
	{
	(i) Let $A_{0}$ be a Borel set on $M$ with $m(A_{0}) = \theta_{0}$ and $m^{+}(A_{0}) = I_{(K,\infty,\infty)}(\theta_{0})$. (See \cite{Mil1} for the existence of $A_{0}$). Let $u, Q,\nu,\{\mu_{I}\}_{I \in Q}$ be the elements of the needle decomposition associated with the function $f(x) = 1_{A}(x) - \theta_{0}$ as in Theorem \ref{th:Ndl}.
	
	By the proof of Theorem \ref{th:iso} explained after Theorem \ref{th:Ndl}, on $\nu$-a.e. needle $I \in Q$, equality holds for the isoperimetric inequality $I_{(I,|\cdot|,\mu_{I})}(\theta_{0}) \geq I_{(K,\infty,\infty)}(\theta_{0})$. By Lemma \ref{lem:1dim}, we obtain $(I, |\cdot|, \mu_{I}) = (\mathbb{R}, |\cdot|, \sqrt{\frac{K}{2\pi}}e^{-\frac{Kx^{2}}{2}}dx)$ as weighted manifolds and $A_{0} \cap I = (-\infty, r_{\theta_{0}}]$ or $[\bar{r_{\theta_{0}}}, \infty)$ by passing to an isometry.
	On each $I \in Q$, we have $|u(x)-u(y)| = d(x,y)$ and hence
	\begin{equation} \label{equ:pleft}
	\int_{I}(u-u_{I})^2 d\mu_{I} = \sqrt{\frac{K}{2\pi}}\int_{\mathbb{R}}x^2 e^{-\frac{Kx^{2}}{2}}dx = \displaystyle\frac{1}{K} 
	\end{equation} where $u_{I}$ stands for the mean value of $u$ on $(I,\mu_{I})$.
	Since $u$ is 1-Lipschitz, we have $1 = |\nabla_{I} u| \leq |\nabla u| \leq 1$. Therefore, $|\nabla u| = 1$ and
	
	\begin{equation} \label{equ:pright}
	\int_{I}|\nabla u|^2 d\mu_{I} = \sqrt{\frac{K}{2\pi}}\int_{\mathbb{R}}e^{-\frac{Kx^{2}}{2}}dx = 1.
	\end{equation}
	By the Fubini theorem and Theorem \ref{th:Ndl}(i), we obtain from \eqref{equ:pleft} and \eqref{equ:pright} that
	
	\begin{align*}
	\int_{M}(u-u_{M})^{2}dm & = \int_{M}u^2dm - \bigg(\int_{M}udm\bigg)^2 \\ & = \int_{Q}\int_{I}u^2d\mu_{I}d\nu - \bigg(\int_{Q}\big(\int_{I}ud\mu_{I}\big)d\nu\bigg)^2 \\ & \geq \int_{Q}\bigg( \int_{I}u^{2}d\mu_{I} - \big(\int_{I}ud\mu_{I}\big)^2 \bigg)d\nu = \int_{Q}\int_{I}(u-u_{I})^2d\mu_{I}d\nu \\
	&= \frac{1}{K} = \displaystyle\frac{1}{K}\int_{M}|\nabla u|^{2}dm.
	\end{align*}
	
	Therefore $v = u - u_{M}$ satisfies the equality in the Poincar\'e inequality and hence becomes an eigenfunction for the first nonzero eigenvalue of the weighted Laplacian on $M$. By Theorem \ref{th:poincare},  $M$ is isometric to the product space $\Sigma^{n-1} \times \mathbb{R}$ as weighted Riemannian manifolds, where $\Sigma^{n-1} = u^{-1}(0)$ is an $(n-1)$-dimensional manifold with $\mathrm{Ric}_{\infty} \ge K$ and $\mathbb{R}$ is equipped with the standard Gaussian measure. Moreover $u(x,t)$ (as the function of the product space) is constant on $\Sigma^{n-1} \times \{t\}$, $u(x,t) = t$. Hence for $\nu$-almost every $I \in Q$, there exists $y_{I} \in \Sigma$ such that $I = \{y_{I}\} \times \mathbb{R}$.
	Then the map $I \mapsto y_{I}$ yields $(Q,\nu) = (\Sigma, m_{\Sigma})$ as measure spaces.
	
	(ii) Now fix $\theta \in (0,1)$ and let $A$ be a mimimizer of the isoperimetric problem with $m(A) = \theta$ and $m^{+}(A) = I_{(K,\infty,\infty)}(\theta)$. Taking the needle decomposition associated with $f=$1$_{A} - \theta$ and the induced splitting $M = \Sigma\times\mathbb{R}$ as in (i), we will show that there exists a half-space $I_{\theta} \subset \mathbb{R}$ such that $A = \Sigma \times I_{\theta}$.
	
	 Assume by contradiction that there exist Borel subsets $Q_{1}, Q_{2}$ of $\Sigma$ with $m_{\Sigma}(Q_{1}) = 1 - m_{\Sigma} (Q_{2}) \in (0, 1)$, such that $A = A_{1} \cup A_{2}$ where $A_{1} := Q_{1} \times (-\infty, r_{\theta}]$ and $A_{2} := Q_{2} \times [\bar{r_{\theta}}, \infty)$.
	 Notice that 
	 \begin{align*}
	m(A^{\epsilon}) - m(A)
	\geq &m(Q_{1} \times [r_{\theta}, r_{\theta} + \epsilon]) + m(Q_{2} \times [\bar{r_{\theta}} - \epsilon, \bar{r_{\theta}}]) \\
     &  + m( (Q_{1}^{\epsilon}\setminus Q_{1})\times (-\infty, r_{\theta}]) + m((Q_{2}^{\epsilon}\setminus Q_{2})\times [\bar{r_{\theta}},\infty)).  
	\end{align*}
	On the one hand,
    \begin{equation}
	   \lim_{\epsilon \downarrow 0} \frac{m(Q_{1} \times [r_{\theta}, r_{\theta} + \epsilon])}{\epsilon} = m_{\Sigma}(Q_{1}) I_{(K,\infty,\infty)}(\theta) ,
	\end{equation} 
	\begin{equation}
	\lim_{\epsilon \downarrow 0} \frac{m(Q_{2} \times [\bar{r_{\theta}} - \epsilon, \bar{r_{\theta}}])}{\epsilon} = m_{\Sigma}(Q_{2}) I_{(K,\infty,\infty)}(\theta).
	\end{equation} 
	On the other hand,
	\begin{align*}
	   \liminf_{\epsilon \downarrow 0} \frac{m( (Q_{1}^{\epsilon}\setminus Q_{1})\times (-\infty,r_{\theta}])}{\epsilon} &= \theta\liminf_{\epsilon \downarrow 0} \displaystyle\frac{m_{\Sigma}(Q_{1}^{\epsilon}) - m_{\Sigma}(Q_{1})}{\epsilon} = \theta m_{\Sigma}^{+}(Q_{1}) \\ & \geq \theta I_{(K,\infty,\infty)}(m_{\Sigma}(Q_{1})) > 0.
	\end{align*}
	Hence $m^{+}(A) > I_{(K,\infty,\infty)}(\theta)$ contradicting the hypothesis. 
	}
	\end{proof}
	
	\begin{remark}
	{
	Since the validity of needle decompositions on RCD$(K,\infty)$ spaces is still unknown, the method we used in this paper can not be extended to that setting. However, the isoperimetric inequality on RCD$(K,\infty)$-spaces was proved in \cite{AM} with the help of the $\Gamma$-calculus and Theorem \ref{th:poincare} was extended to RCD($K,\infty$)-spaces in \cite{GKKO}. 
	} 
	\end{remark}

	\section{Rigidity for isoperimetric inequality of negative effective dimension}%%%%%%%%%
	%%%%%%%%%%%
	
    We next consider the case of $N < -1$. Thanks to Theorem \ref{th:mai}, we can again apply the rigidity of the Poincar\'e inequality. Although the structure of the proof is the same as the case $N= \infty$, there will be several technical difficulties for $N < -1$ and we need more delicate arguments. Firstly, we need to check the equality case of the isoperimetric inequality on $1$-dimensional manifolds to use the needle decomposition method. We use the following lemma to make sure that minimizer sets are half-spaces.
    
    \begin{lemma} \label{lem:minizmizer}
	{
	Let $m = e^{-\psi}dx$ be a strictly log-concave (in the sense that $\psi$ is strictly convex) symmetric probability measure on $\mathbb{R}$. Fix $\theta \in (0,1)$ and if $E$ satisfies $m(E) = \theta$ and $m^{+}(E) = I_{(\mathbb{R},|\cdot|,m)}(\theta)$ then $E$ is necessarily of the form $(-\infty,a]$ or $[b,\infty)$.
	}
	\end{lemma}
	
	\begin{proof}
	{ We can adopt the proof for Gaussian measures in Theorem 3.1 in \cite{CFMP} to the setting of strictly log-concave and symmetric measures.
	Note that $\psi'_{+}$ and $\psi'_{-}$ are monotonically increasing functions and by the symmetry of $m$ (hence $\psi$), we can deduce that $\psi'_{+}(x) + \psi_{+}'(y) > 0$ (resp. $\psi_{-}'(x) + \psi_{-}'(y) < 0$) if $x+y > 0$ (resp. $x+y < 0$).
	
	We call $(a',b')$ a \emph{right-shifted} (resp. \emph{left-shifted}) of an interval $(a,b)$ for $a+b \geq 0$ (resp. $a+b \leq 0$) if $a' > a$ (resp. $a' < a)$ and $m((a',b')) = m((a,b))$. We will show that $m^{+} ((a',b')) < m^{+}((a,b))$ if ($a',b'$) is a right-shifted (resp. left-shifted) of ($a,b$) for $a+b \geq 0$ (resp. $a+b \leq 0$).
	
	Let $a+b \geq 0$. Define $g(\epsilon)$ such that $(a+\epsilon, b+g(\epsilon))$ is a right-shifted of $(a,b)$.
	By definition of $g(\epsilon)$,
	\[ \int_{a}^{b} e^{-\psi(t)}dt = \int_{a+\epsilon}^{b+g(\epsilon)} e^{-\psi(t)}dt. \]
	Hence \[g'(\epsilon)e^{-\psi(b+g(\epsilon))}=e^{-\psi(a+\epsilon)}.\]
	Therefore
	\begin{align*}{
	\frac{d^{+}}{d\epsilon}m^{+}(a+\epsilon,b+g(\epsilon)) &= -\psi'_{+}\big(b+g(\epsilon)\big)g'(\epsilon)e^{-\psi(b+g(\epsilon))} - \psi'_{+}(a+\epsilon)e^{-\psi(a+\epsilon)} \\ &= -e^{-\psi(a+\epsilon)}(\psi_{+}'\big(b+g(\epsilon))+\psi_{+}'(a+\epsilon)\big)< 0.
	}\end{align*}
	The case $a+b \leq 0$ can be proved analogously. 
	
	Hence for a set $F$ of the form $(a,b)\cup(c,d)$ with $a<b<c<d$ and $a+b \geq 0$ (resp. $c+d \leq 0$), we can construct a set $F' = (a',d)$ (resp. $F' = (a,d')$) such that $m(F') = m(F)$ and $m^{+}(F') < m (F)$ by taking $a'$ (resp. $d'$) such that $(a',c)$ is a right-shifted of $(a,b)$ (resp. ($b,d'$) is a left-shifted of $(c,d)$). 
	
	Now let $E = \bigcup_{h \in H} (a_{h},b_{h})$ ($-\infty \leq a_{h} < b_{h} \leq \infty$ and the intervals $(a_{h},b_{h})_{h\in H}$ are mutually disjoint) be a minimizer set of the isoperimetric problem, that is  $m(E) = \theta$ and $m^{+}(E) = I_{(\mathbb{R},|\cdot|,m)}(\theta)$. By repeating the shifting process as above (right-shifted if $a_{h} + b_{h} \geq 0$ and left-shifted if $a_{h} + b_{h} \leq 0$), we can show that $E$ is necessarily of the form $(-\infty,a]$ or $[b,\infty)$ or $(-\infty,a]\cup[b,\infty)$.
	
	Suppose $E = (-\infty,a]\cup[b,\infty)$. Let $E_{1} = \mathbb{R} \setminus E = (a,b)$, we have $m(E_{1}) = 1-\theta$ and $m^{+} (E_{1}) = m^{+}(E)$. By taking left-shifted or right-shifted of $E_{1}$, we find a half-line $E_{2}$ such that $m(E_{2}) = m(E_{1}) = 1 -\theta$ and $m^{+} (E_{2}) < m^{+}(E)$.  Let $E_{3} = \mathbb{R} \setminus E_{2}$, we have $m(E_{3}) = \theta$ and $m^{+} (E_{3}) < m^{+}(E)$. This is a contradiction. Hence $E$ is necessarily of the form $(-\infty,a]$ or $[b,\infty)$.
	}
	\end{proof}

	\begin{lemma}($1$-dimensional case)\label{lem:1dim-neg}
	{
	Let $(M,g,m)$ be a weighted Riemannian manifold of dimension $1$ satisfying $\emph{Ric}_{N} \geq K > 0$ for $N < -1$ and suppose $m(M) = 1$. Assume that there exists $\theta_{0} \in (0,1)$ with $I_{(M,g,m)}(\theta_{0}) = I_{(K,N,\infty)}(\theta_{0})$. Then 
	
	\[ (M,g,m) = \bigg( \mathbb{R},|\cdot|, m_{K,N}^{-1} \cosh^{N-1}\bigg( {\sqrt{\frac{K}{1-N}}}x\bigg) dx \bigg) \] 
	as weighted manifolds where $m_{K,N} = \int_{-\infty}^{\infty}\cosh^{N-1}\bigg( {\sqrt{\frac{K}{1-N}}}x\bigg)dx$. Moreover, if a Borel set $A$ satisfies $m^{+}(A) = I_{(K,N,\infty)}(m(A))$ then $A$ is $(-\infty,a]$ or $[b,\infty)$ up to a difference of an $m$-negligible set.
	}
	\end{lemma}
	
	\begin{proof}
	{
	We will prove $I_{(K,N,D)}(\theta_{0}) > I_{(K,N,\infty)}(\theta_{0})$ in the appendix. Hence $M$ is noncompact and isometric to $\mathbb{R}$. We now recall the proof of the isoperimetric inequality in \cite{Mil1,Mil2}.

	Since $M$ is $1$-dimensional and $\mathrm{Ric}_{N} \geq K$, $\psi$ satisfies the $(K,N-1)$-convexity condition in \cite{Oneg}: $\psi'' - \frac{(\psi')^{2}}{N-1} \geq K$ in the weak sense. Hence by \cite[(2.5)]{Oneg}
	
	\begin{equation} \label{equ:key-neg1}
	{e^{-\psi(x+t)} \leq e^{-\psi(x)}J_{\psi'_{+}(x)}}(t) \qquad (t>0),
	\end{equation} and
	
	\begin{equation} \label{equ:key-neg2}
	{e^{-\psi(x+t)} \leq e^{-\psi(x)}J_{-\psi'_{-}(x)}}(t) \qquad (t<0).
	\end{equation}
	Here we put $\sigma := K/(1-N)$ and denote $\big( \cosh(\sqrt{\sigma} t) + \frac{H \sinh(\sqrt{\sigma}t)}{(N-1)\sqrt{\sigma}} \big)_{+}^{N-1}$ by $J_{H}(t)$.
	Let $A$ be a minimizer for the isoperimetric problem. We may  assume without loss of generality that $M\setminus A$ does not include isolated points, since those will not influence $m(A)$ and $m^{+}(A)$. We have $m^{+}(A) = \sum_{x \in \partial A} e^{-\psi(x)}$. We will show that for all $x \in \partial A$, $\psi(x)$ is differentiable and $H(x)$ is constant where $H(x)$ denotes $\psi' (x)$ (resp. $-\psi' (x)$) if $x$ is a right (resp. left) boundary point. 
	
	Let $(a,b)$ be a connected component of $A$ and $g(\epsilon)$ is defined as in Lemma \ref{lem:minizmizer} to construct a right-shifted of $(a,b)$. Recall  \[g'(\epsilon)e^{-\psi(b+g(\epsilon))}=e^{-\psi(a+\epsilon)},\]
	and
    \[
	\frac{d^{+}}{d\epsilon}m^{+}(a+\epsilon,b+g(\epsilon)) = -e^{-\psi(a+\epsilon)}(\psi_{+}'\big(b+g(\epsilon))+\psi_{+}'(a+\epsilon)\big)
	\] which is nonnegative by the isoperimetric inequality.
	Letting $\epsilon$ go to $0$, we obtain $\psi'_{+}(b) + \psi'_{+} (a) \leq 0$. Using a similar argument to the left-shifted of $(a,b)$ implies $\psi'_{-}(b) + \psi'_{-} (a) \geq 0$. The $(K,N)$-convexity implies that $\psi'_{+} \geq \psi'_{-}$, hence $\psi'_{+}(b) = -\psi'_{-} (a) = \psi'_{-}(b) = -\psi'_{+} (a)$. Note that $M\setminus A$ is the isoperimetric minimizer of parameter $1-\theta$, we can conclude that $H(x)$ is a constant and we write $H(x)$ by $H$.
	
	From \eqref{equ:key-neg1} and \eqref{equ:key-neg2}
	\[ m(A^{r}) - m(A)  \leq m^{+} (A)\int_{0}^{r} J_{H}(t)dt. \]
	Letting $r$ go to $\infty$, we have
	\begin{equation}\label{equ:infneg}
	1 - \theta \leq m^{+}(A) \int_{0}^{\infty}  J_{H}(t)dt.
	\end{equation}
	Using a similar argument for $M\setminus A$, we also obtain
	
	\begin{equation}\label{equ:minusinfneg}
	\theta \leq m^{+}(A) \int_{0}^{\infty} J_{-H}(t)dt = m^{+}(A) \int_{-\infty}^{0} J_{H}(t)dt.
	\end{equation}
	Hence
	\begin{equation}\label{equ:maxmin}
	I_{(M,g,m)} (\theta)\geq \inf_{H\in \mathbb{R}} \max\bigg\{\frac{1-\theta}{\int_{0}^{\infty} J_{H}(t) dt}, \frac{\theta}{\int_{-\infty}^{0} J_{H}(t)dt}\bigg\}.
	\end{equation} When $H$ varies from $-\infty$ to $\infty$, the first (resp. second) term in the right hand side of (\ref{equ:maxmin}) varies monotonically from $\infty$ to $0$ (resp. $0$ to $\infty$). Therefore the infimum over $H\in \mathbb{R}$ is attained at the unique point $H_{\theta}$ such that both terms coincide:
	
	\begin{equation}
	\frac{\int_{0}^{\infty} J_{H_{\theta}}(t) dt}{1-\theta} = \frac{\int_{-\infty}^{0} J_{H_{\theta}}(t)dt}{\theta} = \int_{-\infty}^{\infty} J_{H_{\theta}}(t)dt.
	\end{equation} 
	Put $\beta = \frac{H_{\theta}}{(N-1)\sqrt{\sigma}}$ and $\alpha = \tanh^{-1}(\beta) \in \mathbb{R}$.
	Since $J_{H_{\theta}}(t)$ is integrable, we indeed have $|\beta| < 1$ (see \cite[\S 4]{Mil2}) and 
	
	\[ J_{H_{\theta}}(t) = \frac{\cosh ^{N-1} (\alpha + \sqrt{\sigma}t)}{\cosh^{N-1}(\alpha)}. \]
	Hence
	
	\[ \theta = \frac{\int_{-\infty}^{0}\cosh ^{N-1} (\alpha + \sqrt{\sigma}t)dt}{\int_{-\infty}^{\infty}\cosh ^{N-1} (\alpha + \sqrt{\sigma}t)dt} = \frac{\int_{-\infty}^{\alpha/\sqrt{\sigma}}\cosh ^{N-1} (\sqrt{\sigma}s)ds}{\int_{-\infty}^{\infty}\cosh ^{N-1} (\sqrt{\sigma}s)ds}.\]
	Thus $c(\theta) = \frac{\alpha}{\sqrt{\sigma}}$ and
	
	\begin{align*} I_{(M,g,m)} (\theta) &\geq \bigg( \int_{-\infty}^{\infty}\frac{\cosh ^{N-1} (\alpha + \sqrt{\sigma}t)}{\cosh^{N-1}(\alpha)}dt \bigg)^{-1} = \frac{\cosh^{N-1}(\sqrt{\sigma}c(\theta))}{{\int_{-\infty}^{\infty}\cosh ^{N-1} (\sqrt{\sigma}s)ds}} \\&= I_{K,N,\infty} (\theta). \end{align*}
	
	We consider the case of equality. Let $f_{N} = e^{\psi/(1-N)}$. By section 2.1 in \cite{Oneg}, the equality of \eqref{equ:key-neg1} and \eqref{equ:key-neg2} imply that $f_{N}'' = \frac{K}{1-N} f_{N} (x)$. Hence $f_{N} = a \cosh (\sqrt{\sigma}x) + b\sinh(\sqrt{\sigma}x)$. Since $e^{-\psi} = f_{N}^{N-1}$ is positive and integrable, we have $a>0$ and $|b/a| < 1$. Put $\gamma = \tanh^{-1} (b/a)$ and $k = a/\cosh{\gamma}$, then $e^{-\psi}(x) = (k \cosh(\gamma + \sqrt{\sigma}x))^{N-1}$.  Therefore $(M,g,m) = \bigg( \mathbb{R},|\cdot|, m_{K,N}^{-1}\cosh^{N-1}\big( {\sqrt{\sigma}x}\big) \,dx \bigg)$ as weighted manifolds.  
	
	Put $h(x) = - \log (\cosh^{N-1}({\sqrt{\sigma}}x))$. We have
	\[ h''(x) = K \frac{1}{\cosh^{2} ({\sqrt{\sigma}}x)} > 0. \]
	Therefore $m$ is a strictly log-concave measure. Note that it is also symmetric. Hence by Lemma \ref{lem:minizmizer}, if a Borel set $A$ satisfies $m^{+}(A) = I_{(K,N,\infty)}(m(A))$ then $A = (-\infty,a]$ or $[b,\infty)$.
	}
	\end{proof}
	
	Now we turn to the main theorem. 
	
	\begin{theorem}(High dimensional case)\label{th:ndimneg}
	{
	Let $(M,g,m)$ be a weighted Riemannian manifold of dimension $n \geq 2$ satisfying $\emph{Ric}_{N} \geq K > 0$ for $N < -1$ and suppose $m(M) = 1$. Assume that there exists $\theta_{0} \in (0,1)$ with $I_{(M,g,m)}(\theta_{0}) = I_{(K,N,\infty)}(\theta_{0})$. Then the following hold:
	\begin{enumerate}[{\rm (i)}]
	        \item $M$ is isometric to the warped product
			\[ \mathbb{R} \times_{\cosh(\sqrt{K/(1-N)}t)} \Sigma
			=\bigg( \mathbb{R} \times \Sigma,
			 dt^2 +\cosh^2\bigg( \sqrt{\frac{K}{1-N}}t \bigg) \cdot g_{\Sigma} \bigg) \]
			and the measure $m$ is written through the isometry as
			\[ m(dtdx) =\cosh^{N-1}\bigg( \sqrt{\frac{K}{1-N}}t \bigg) \,dt \,m_{\Sigma}(dx), \]
			where $\Sigma^{n-1}$ is an $(n-1)$-dimensional weighted Riemannian manifold with $\mathrm{Ric}_{N-1} \geq K(2-N)/(1-N)$.
			\item If a Borel set $A$ satisfies $m^{+}(A) = I_{(K,N,\infty)}(m(A))$ then $A$ is a half-space of $M$ (the precise meaning will be given in the proof). 
		\end{enumerate}
		
	}
	\end{theorem}
	
	\begin {proof}
	{
	(i) Let $A_{0}$ be a Borel set on $M$ with $m(A_{0}) = \theta_{0}$ and $m^{+}(A_{0}) = I_{(K,N,\infty)}(\theta_{0})$. (See \cite{Mil1} for the existence of $A_{0}$). By Theorem 2.3 in \cite{Mil1}, we can choose $A_{0}$ such that the regular part $\partial_{r}A_{0}$ of the boudary is an open $C^{\infty}$-hypersurface up to Hausdorff codimension 7. Moreover, on $\partial_{r}A_{0}$, we can define the outward normal unit vector $n(x)$. 
	
	Let $u, Q,\nu,\{\mu_{I}\}_{I \in Q}$ be the elements of the needle decomposition associated with the function $f(x) = 1_{A_{0}}(x) - \theta_{0}$ as in Theorem \ref{th:Ndl}. On $\nu$-a.e. needle $I \in Q$, equality holds for the isoperimetric inequality $I_{(I,|\cdot|,\mu_{I})}(\theta_{0}) \geq I_{(K,N,\infty)}(\theta_{0})$. Moreover, by Lemma \ref{lem:1dim-neg}, we obtain 
	\[(I, |\cdot|, \mu_{I}) = \bigg(\mathbb{R}, |\cdot|, m_{K,N}^{-1} \cosh^{N-1}\bigg( {\sqrt{\frac{K}{1-N}}}x\bigg) dx \bigg) \] as weighted manifolds and $A_{0} \cap I = (-\infty, r_{\theta_{0}}]$ or $[\bar{r_{\theta_{0}}}, \infty)$ by passing to an isometry. By Corollary 2.28 in \cite{Kl}, $u$ can be chosen to be a $C^{1,1}$-function, it means $u$ is $C^{1}$ and $\nabla u$ is locally Lipschitz. By Lemma 10 in \cite{FM},  for any needle $I$ and a point $x\in I$, the function $u$ is differentiable at $x$ and $\nabla u(x)$ is a unit vector tangent to $I$. 
	
	Put $\sigma = K/(1-N)$. On each $I \in Q$, we have $|u(x)-u(y)| = d(x,y)$ hence we can consider $u(x)$ as the function $t \mapsto t+c$ for a constant $c$ on $\mathbb{R}$. Assume $\partial_{r}A_{0}$ is not perpendicular to $\nabla u$, precisely, the Hausdorff measure $H^{n-1}(\{ \langle \nabla u,n\rangle \leq 1-\rho\})$ is positive for some $\rho > 0$. It means the $\epsilon$-neighborhood of $A_{0}$ in $M$ contains a nonzero-measure sets that is not included in the union of $\epsilon$-neighborhoods of $A_{0}\cap I$ in $I$. Note that on $\nu$-a.e. $I$, $m(x) > 0$ for $x\in I$. Therefore the equality in \eqref{equ:orthogonal} can not hold or, in other words, $m^{+}(A_{0}) > I_{(K,N,\infty)}(\theta_{0})$, this is a contradiction. Hence $\partial_{r}A_{0}$ is perpendicular to $\nabla u$, and $c$ is constant on  $\nu$-a.e. $I$. Let $u_{I}$ stand for the mean value of $u$ on $(I,\mu_{I})$. Then $u_{I}$ is constant for $\nu$-a.e. $I$ and $u_{M} = u_{I}$ for $\nu$-a.e. $I$.
	
	Letting $v = \sinh (\sqrt{\sigma}(u - u_{M}))$, we will prove that $v$ gives the equality of the Poincar\'e inequality. Since $\nabla u(x)$ is a unit vector tangent to $I$, $|\nabla u| = 1$ and $|\nabla v| = \sqrt{\sigma}\cosh({\sqrt{\sigma}}(u-u_{M}))$. We have
	\[
	I_{1} := \int_{I} |\nabla v|^{2} d\mu_{I} =  m_{K,N}^{-1} \sigma\int_{\mathbb{R}} \cosh^{N+1}(\sqrt{\sigma}x)dx.
	\]
	Letting $v_{I}$ stand for the mean value of $v$ on $(I,\mu_{I})$, we have
	\begin{align*}
	    v_{I} &= \int_{I} \sinh (\sqrt{\sigma}(u-u_{M})) d\mu_{I} 
	          \\&= \int_{I} \sinh (\sqrt{\sigma}(u-u_{I})) d\mu_{I} 
	          \\&= m_{K,N}^{-1} \int_{\mathbb{R}} \sinh(\sqrt{\sigma}x)\cosh^{N-1}(\sqrt{\sigma}x)dx = 0.
	\end{align*}
	Hence
	\[I_{2} := \int_{I} (v-v_{I})^{2} d\mu_{I} =  m_{K,N}^{-1} \int_{\mathbb{R}}\sinh^{2}(\sqrt{\sigma}x)\cosh^{N-1}(\sqrt{\sigma}x)dx.\]
	By integration by parts,
	\begin{align*}
	I_{1} &= m_{K,N}^{-1}\sigma\int_{\mathbb{R}}\cosh(\sqrt{\sigma}x)\cdot\cosh^{N}(\sqrt{\sigma}x)dx \\&= -m_{K,N}^{-1}\sigma\int_{\mathbb{R}}\frac{\sinh(\sqrt{\sigma}x)}{\sqrt{\sigma}}\cdot N\sqrt{\sigma}\cosh^{N-1}(\sqrt{\sigma}x)\sinh(\sqrt{\sigma}x)dx
	\\&= -N\sigma I_{2} = \frac{KN}{N-1} I_{2}.
	\end{align*}
	Therefore
	\begin{align*}
	\int_{M}(v-v_{M})^{2}dm & = \int_{M}v^2dm - \bigg(\int_{M}vdm\bigg)^2 \\ & = \int_{Q}\int_{I}v^2d\mu_{I}d\nu - \bigg(\int_{Q}\big(\int_{I}vd\mu_{I}\big)d\nu\bigg)^2 \\ &= \int_{Q}\int_{I}v^2d\mu_{I}d\nu = \displaystyle\frac{N-1}{KN} \int_{Q}\int_{I}|\nabla v|^{2}d\mu_{I}d\nu\\
	&= \displaystyle\frac{N-1}{KN}\int_{M}|\nabla v|^{2}dm.
	\end{align*}
	Therefore $v$ satisfies the equality in the Poincar\'e inequality and hence becomes an eigenfunction for the first nonzero eigenvalue of the weighted Laplacian on $M$. By Theorem \ref{th:mai}, $M$ is isometric to the warped product \[\mathbb{R} \times_{\cosh(\sqrt{K/(1-N)}t)} \Sigma
			=\bigg( \mathbb{R} \times \Sigma,
			 dt^2 +\cosh^2\bigg( \sqrt{\frac{K}{1-N}}t \bigg) \cdot g_{\Sigma} \bigg), \] and the measure $m$ is written through the isometry as
			\[ m(dtdx) =\cosh^{N-1}\bigg( \sqrt{\frac{K}{1-N}}t \bigg) \,dt \,m_{\Sigma}(dx),\] 
			where $\Sigma^{n-1}  = \{v^{-1}(0)\}$ is an $(n-1)$-dimensional weighted Riemannian manifold.
	
	(ii) Now fix $\theta \in (0,1)$ and let $A$ be a mimimizer of the isoperimetric problem with $m(A) = \theta$ and $m^{+}(A) = I_{(K,N,\infty)}(\theta)$. Take the needle decomposition associated with $f=$1$_{A} - \theta$ as in (i), we will show that there exists a half-space $I_{\theta} \subset \mathbb{R}$ such that $A = I_{\theta} \times \Sigma$.
	
	 Assume by contradiction that there exist Borel subsets $Q_{1}, Q_{2}$ of $\Sigma$ with $m_{\Sigma}(Q_{1}) = 1 - m_{\Sigma} (Q_{2}) \in (0, 1)$, such that $A = A_{1} \cup A_{2}$ where $A_{1} := Q_{1} \times (-\infty, r_{\theta}]$ and $A_{2} := Q_{2} \times [\bar{r_{\theta}}, \infty)$.
	 Notice that for $(p,t),(q,t) \in \Sigma \times \mathbb{R}$, we have $d((p,t),(q,t)) = \cosh(\sqrt{\sigma}t) d_{\Sigma}(p,q)$. Fix $b \in (-\infty,r_{\theta})$ and $k$ the maximal value of $\cosh (\sqrt{\sigma}x)$ on $[b,r_{\theta}]$. Then
	 \begin{align*}
	m(A^{\epsilon}) - m(A)
	\geq &m(Q_{1} \times [r_{\theta}, r_{\theta} + \epsilon]) + m(Q_{2} \times [\bar{r_{\theta}} - \epsilon, \bar{r_{\theta}}]) \\
     &  + m( (Q_{1}^{\epsilon/k}\setminus Q_{1})\times (b, r_{\theta}]) .  
	\end{align*}
	On the one hand
    \begin{equation}
	   \lim_{\epsilon \downarrow 0} \frac{m(Q_{1} \times [r_{\theta}, r_{\theta} + \epsilon])}{\epsilon} = m_{\Sigma}(Q_{1}) I_{(K,N,\infty)}(\theta) ,
	\end{equation} 
	\begin{equation}
	\lim_{\epsilon \downarrow 0} \frac{m(Q_{2} \times [\bar{r_{\theta}} - \epsilon, \bar{r_{\theta}}])}{\epsilon} = m_{\Sigma}(Q_{2}) I_{(K,N,\infty)}(\theta).
	\end{equation} 
	Put $c := \frac{1}{m_{K,N}}\int_{b}^{r_{\theta}}\cosh^{N-1}(\sqrt{\sigma}t)dt \geq \frac{k^{N-1}}{m_{K,N}}(r_{\theta}-b)$, we have
	\begin{align*}
	   \liminf_{\epsilon \downarrow 0} \frac{m( (Q_{1}^{\epsilon/k}\setminus Q_{1})\times (b,r_{\theta}])}{\epsilon} &= c\liminf_{\epsilon \downarrow 0} \displaystyle\frac{m_{\Sigma}(Q_{1}^{\epsilon/k}) - m_{\Sigma}(Q_{1})}{\epsilon} \\&= \frac{c}{k} m_{\Sigma}^{+}(Q_{1}) \\ & \geq \frac{c}{k} I_{(K(2-N)/(1-N),N-1,\infty)}(m_{\Sigma}(Q_{1})) > 0.
	\end{align*}
	Hence $m^{+}(A) > I_{(K,N,\infty)}(\theta)$ contradicting the hypothesis. 
	}
	\end {proof}
	
	\begin{remark}
	{
	The isoperimetric inequality can be extended to a quantitative version on Gaussian spaces (see \cite{CFMP} and \cite{MN,El,BBJ} for dimension-free estimates) and metric measure spaces of positive effective dimension (see \cite{CMM}). It was showed that the volume of the symmetric difference between a set and the isoperimetric minimizer is bounded from above by a functional form of the isoperimetric deficit of this set. When $N < 0$, needle decomposition method might be helpful to study a similar estimate.
	}
	\end{remark}

\begin{appendices}
\section{}
In the appendix, we will prove that $I_{K,N,D}(\theta) > I_{K,N,\infty}(\theta)$ for all $\theta \in (0,1)$, $N < 0$ and $D < \infty$. We abbreviate in this appendix $I(\theta) := I_{K,N,\infty}(\theta)$.

\begin{lemma}\label{lema1}
$K_{3,D}(\theta) > I_{K,N,\infty}(\theta)$ for all $\theta \in (0,1)$, $N < 0$ and $D < \infty$.
\end{lemma}
\begin{proof}
By the definitions of $c(\theta)$,$d_{3}(\theta)$ in $\S$2.2, we have
	$1 = I(\theta)c'(\theta) = K_{3,D}(\theta)d_{3}'(\theta)$. Therefore
		\[ I'(\theta) = (N-1)\sqrt{\sigma}\tanh(\sqrt{\sigma}c(\theta)),\]
		\[ K_{3,D}'(\theta) = (N-1)\sqrt{\sigma}.\] Put $h(\theta) = K_{3,D}(\theta)-I(\theta)$. Since $|\tanh x| < 1$, we have $h'(\theta) < 0$. Hence $h(\theta) \geq h(1) = \frac{e^{(N-1)\sqrt{\sigma}D}}{\int_{0}^{D} e^{(N-1)\sqrt{\sigma}s}ds} > 0$.
\end{proof}

\begin{lemma}\label{lemma2}
$K_{2,D}(\theta) > I_{K,N,\infty}(\theta)$ for all $\theta \in (0,1)$, $N < 0$ and $D < \infty$.
\end{lemma}
\begin{proof}

Put $\varphi(t) := \sinh^{N-1}(\sqrt{\sigma}t)$. For $t>0$, we have
\[ \varphi'(t) = (N-1)\sqrt{\sigma}\varphi(t)\coth(\sqrt{\sigma}t) < 0.\] Fix $\xi>0$ and set $g_{\xi}(\theta) := \frac{\varphi(d_{2,\xi}(\theta))}{\int_{\xi}^{\xi + D} \varphi(s)ds} > 0$. By the definitions of $c(\theta)$, $d_{2,\xi}(\theta)$, we have
	$1 = I(\theta)c'(\theta) = g_{\xi}(\theta)d_{2,\xi}'(\theta)$. Therefore
	\[ I'(\theta) = (N-1)\sqrt{\sigma}\tanh(\sqrt{\sigma}c(\theta)),\]
	\[ g_{\xi}'(\theta) = (N-1)\sqrt{\sigma}\coth(\sqrt{\sigma}d_{2,\xi}(\theta)).\] Put $h_{\xi}(\theta) = g_{\xi}(\theta)-I(\theta)$. Since $|\tanh x| < 1$ and $\coth y > 1$ for $y>0$, $h_{\xi}(\theta) \geq h_{\xi}(1)$. Note that $I(1) = 0$, hence $h_{\xi}(\theta) > 0$ for each $\xi > 0$. Fix $\theta$ and put $m(\xi) = \frac{\int_{\xi}^{\xi + D} \varphi(s)ds}{\varphi(\xi+D)} > 0$. Denote $\cosh(\sqrt{\sigma} D)$, $\sinh(\sqrt{\sigma} D)$ by $c_{D}$ and $s_{D}$. We have
\begin{align*}
\lim_{\xi \rightarrow \infty} m(\xi) 
&= \lim_{\xi \rightarrow \infty}  \frac{\varphi(\xi+D)-\varphi(\xi)}{\varphi'(\xi+D)} =   \lim_{\xi \rightarrow \infty}\frac{\varphi(\xi+D)}{\varphi'(\xi+D)} - \lim_{\xi \rightarrow \infty} \frac{\varphi(\xi)}{\varphi(\xi+D)} \frac{\varphi(\xi+D)}{\varphi'(\xi+D)}
\\&= \frac{1}{(N-1)\sqrt{\sigma}}-\frac{1}{(N-1)\sqrt{\sigma}}\lim_{\xi \rightarrow \infty} \frac{\varphi(\xi)}{\varphi(\xi+D)}
\\ &= \frac{1}{(N-1)\sqrt{\sigma}}  - \frac{1}{(N-1)\sqrt{\sigma}}\lim_{\xi \rightarrow \infty} \big( c_{D} + s_{D}\coth(\sqrt{\sigma}\xi)\big)^{1-N} 
\\& = \frac{1}{(N-1)\sqrt{\sigma}} \big(1- (c_{D}+s_{D})^{1-N}\big).
\end{align*} 
Note that $\varphi(d_{2,\xi}(\theta)) \geq \varphi(\xi + D)$, hence 
\[\lim_{\xi \rightarrow \infty} h_{\xi}(\theta) \geq \lim_{\xi \rightarrow \infty} h_{\xi}(1) \geq \lim_{\xi \rightarrow \infty} (m(\xi))^{-1} > 0.\]
    
Now we consider the other direction $\lim_{\xi \rightarrow 0} h_{\xi}(\theta)$. We also define $\bar{d}(\xi)$ as
\[ \theta = \frac{\int_{\xi}^{\bar{d}(\xi)} (\sqrt{\sigma}s)^{N-1}ds}{\int_{\xi}^{\xi+D} (\sqrt{\sigma}s)^{N-1}ds} = \frac{\bar{d}^{N}_{\xi}-\xi^{N}}{(\xi+D)^N-\xi^{N}}.\]
Put $f(\Delta) := \frac{\int_{\xi}^{\xi+\Delta} \varphi(s)ds}{\int_{\xi}^{\xi+\Delta} (\sqrt{\sigma}s)^{N-1}ds} \geq 0$. We have
\begin{align*}
    f'(\Delta) &= f(\Delta)\bigg( \frac{\varphi(\xi+\Delta)}{\int_{\xi}^{\xi+\Delta} \varphi(s)ds} - \frac{(\sqrt{\sigma}(\xi+\Delta))^{N-1}}{\int_{\xi}^{\xi+\Delta} (\sqrt{\sigma}s)^{N-1}ds}\bigg) 
    \\&= \frac{f(\Delta)(\sqrt{\sigma}(\xi+\Delta))^{N-1}}{\int_{\xi}^{\xi+\Delta} \varphi(s)ds}\bigg( \frac{\varphi(\xi+\Delta)}{(\sqrt{\sigma}(\xi+\Delta))^{N-1}} - \frac{\int_{\xi}^{\xi+\Delta} \varphi(s)ds}{\int_{\xi}^{\xi+\Delta} (\sqrt{\sigma}s)^{N-1}ds} \bigg).
\end{align*}
By Cauchy's mean value theorem, there exists $c\in (\xi,\xi+\Delta)$ such that $\frac{\int_{\xi}^{\xi+\Delta} \varphi(s)ds}{\int_{\xi}^{\xi+\Delta} (\sqrt{\sigma}s)^{N-1}ds} = \frac{\varphi(c)}{(\sqrt{\sigma}c)^{N-1}}$. Note that $\frac{\sinh x}{x}$ is a monotone increasing function when $x > 0$, we have $\frac{\varphi(x)}{(\sqrt{\sigma}x)^{N-1}}$ is monotone decreasing. Hence $f'(\Delta) < 0$ and $f$ is monotone decreasing. Therefore

\[ \frac{\int_{\xi}^{\bar{d}(\xi)} (\sqrt{\sigma}s)^{N-1}ds}{\int_{\xi}^{d_{2,\xi}(\theta)} \varphi(s)ds} = \frac{\int_{\xi}^{\xi + D} (\sqrt{\sigma}s)^{N-1}ds}{\int_{\xi}^{\xi +D} \varphi(s)ds} \geq \frac{\int_{\xi}^{d_{2,\xi}(\theta)} (\sqrt{\sigma}s)^{N-1}ds}{\int_{\xi}^{d_{2,\xi}(\theta)} \varphi(s)ds}.\] Thus we can deduce that $d_{2,\xi}(\theta) \leq \bar{d}(\xi) = \big( \theta (\xi+D)^{N} + (1-\theta)\xi^{N}\big)^{1/N}$. Hence 
\begin{align*} 
\lim_{\xi \rightarrow 0_{+}} h_{\xi}(\theta) &\geq \lim_{\xi \rightarrow 0_{+}} \frac{\varphi(\bar{d}(\xi))}{\int_{\xi}^{\xi + D} \varphi(s)ds} - I(\theta) \\&= \lim_{\xi \rightarrow 0_{+}} \frac{(N-1)\sqrt{\sigma}\varphi(\bar{d}(\xi))\coth({\sqrt{\sigma}}\bar{d}(\xi))\bar{d}'(\xi)}{\varphi(\xi+D)-\varphi(\xi)} - I(\theta)\\&= \lim_{\xi \rightarrow 0_{+}} \frac{(N-1)\sqrt{\sigma}^{N}\bar{d}^{N-1}(\xi)\bar{d}'(\xi)}{\sinh({\sqrt{\sigma}}\bar{d}(\xi))\big(\varphi(D)-\varphi(\xi)\big)} 
- I(\theta)\\&= \lim_{\xi \rightarrow 0_{+}} \frac{1}{\sinh({\sqrt{\sigma}}\bar{d}(\xi))}\cdot \frac{(N-1)\sqrt{\sigma}^{N}\big(\theta (\xi+D)^{N-1} + (1-\theta)\xi^{N-1}\big)}{\varphi(D)-\varphi(\xi)} - I(\theta)\\&= \lim_{\xi \rightarrow 0_{+}} \frac{1}{\sinh({\sqrt{\sigma}}\bar{d}(\xi))}\cdot\frac{(1-N)\sqrt{\sigma}\big(\theta (\xi+D)^{N-1} + (1-\theta)\xi^{N-   1}\big)}{\xi^{N-1}}- I(\theta)
\\&= \infty. 
\end{align*}
Therefore $K_{2,D}(\theta) - I(\theta) = \inf_{\xi>0} h_{\xi}(\theta) > 0$.
\end{proof}

\begin{lemma}\label{lema3}
$K_{1,D}(\theta) > I_{K,N,\infty}(\theta)$ for all $\theta \in (0,1)$, $N < 0$ and $D < \infty$.
\end{lemma}
\begin{proof}
Put $\varphi(t) := \cosh^{N-1}(\sqrt{\sigma}t)$. We have
\[ \varphi'(t) = (N-1)\sqrt{\sigma}\varphi(t)\tanh(\sqrt{\sigma}t).\]
Put $g_{\xi}(\theta) :=  \frac{\varphi(d_{1,\xi}(\theta))}{\int_{\xi}^{\xi + D} \varphi(s)ds}$. By the definitions of $c(\theta)$,$d_{1.\xi}(\theta)$ in $\S$2.2, we have
	$1 = I(\theta)c'(\theta) = g_{\xi}(\theta)d_{1,\xi}'(\theta)$. Therefore
	\[ I'(\theta) = (N-1)\sqrt{\sigma}\tanh(\sqrt{\sigma}c(\theta)),\]
	\[ g_{\xi}'(\theta) = (N-1)\sqrt{\sigma}\tanh(\sqrt{\sigma}d_{1,\xi}(\theta)).\] Put $h_{\xi}(\theta) = g_{\xi}(\theta)-I(\theta)$. Since $\tanh(x)$ is monotone increasing we have $h_{\xi}'(\theta_{0}) = 0$ if and only if $c(\theta_{0}) = d_{1,\xi}(\theta_{0})$.  Since $I(0) = I(1) = 0$, we can deduce  $h_{\xi}(\theta) \geq \min\{g_{\xi}(0),g_{\xi}(1), h_{\xi}(\theta_{0})$ where $c(\theta_{0}) = d_{1,\xi}(\theta_{0})\}$.

Put $m_{0}(\xi) := (g_{\xi}(0))^{-1} = \frac{\int_{\xi}^{\xi + D} \varphi(s)ds}{\varphi(\xi)}$. We have

\begin{align*}
    m'_{0}(\xi) &= \frac{\varphi(\xi+D)-\varphi(\xi)}{\varphi(\xi)} - m_{0}(\xi)\frac{\varphi'(\xi)}{\varphi(\xi)} \\ &= m_{0}(\xi) \bigg( \frac{\varphi(\xi+D)-\varphi(\xi)}{\int_{\xi}^{\xi + D} \varphi(s)ds} - (N-1)\sqrt{\sigma}\tanh(\sqrt{\sigma}\xi)\bigg).
\end{align*}
By Cauchy's mean value theorem, there exists $c\in (\xi,\xi+D)$ such that $\frac{\varphi(\xi+D)-\varphi(\xi)}{\int_{\xi}^{\xi + D} \varphi(s)ds} = \frac{\varphi'(c)}{\varphi(c)} = (N-1)\sqrt{\sigma}\tanh(\sqrt{\sigma}c)$. Note that $\tanh x$ is a monotone increasing function, we have $m_{0}(\xi)$ is a monotone decreasing function. Hence $g_{\xi}(0) \geq \lim_{\xi \rightarrow -\infty} \frac{\varphi(\xi)}{\int_{\xi}^{\xi + D} \varphi(s)ds}$. Denote $\cosh(\sqrt{\sigma} D)$, $\sinh(\sqrt{\sigma} D)$ by $c_{D}$ and $s_{D}$.  We have
\begin{align*}
\lim_{\xi \rightarrow -\infty} m_{0}(\xi) 
&= \lim_{\xi \rightarrow -\infty}  \frac{\varphi(\xi+D)-\varphi(\xi)}{\varphi'(\xi)} =  \lim_{\xi \rightarrow -\infty} \frac{\varphi(\xi + D)}{\varphi(\xi)} \frac{\varphi(\xi)}{\varphi'(\xi)} - \lim_{\xi \rightarrow -\infty}\frac{\varphi(\xi)}{\varphi'(\xi)} 
\\&= \frac{-1}{(N-1)\sqrt{\sigma}}\lim_{\xi \rightarrow -\infty} \frac{\varphi(\xi + D)}{\varphi(\xi)} - \frac{-1}{(N-1)\sqrt{\sigma}}
\\ &= \frac{-1}{(N-1)\sqrt{\sigma}}\lim_{\xi \rightarrow -\infty} \big( c_{D} + s_{D}\tanh(\sqrt{\sigma}\xi)\big)^{N-1} - \frac{-1}{(N-1)\sqrt{\sigma}} 
\\& = \frac{-1}{(N-1)\sqrt{\sigma}} \big( (c_{D}-s_{D})^{N-1} - 1 \big).
\end{align*} 
Hence $g_{\xi}(0) \geq \displaystyle\frac{(1-N)(\sqrt{\sigma})}{( (c_{D}-s_{D})^{N-1} - 1 )} > 0$.

On the other hand, put $m_{1}(\xi) := (g(1))^{-1} = \frac{\int_{\xi}^{\xi + D} \varphi(s)ds}{\varphi(\xi+D)}$. We have

\begin{align*}
    m'_{1}(\xi) &= \frac{\varphi(\xi+D)-\varphi(\xi)}{\varphi(\xi+D)} - m_{1}(\xi)\frac{\varphi'(\xi +D)}{\varphi(\xi+D)} \\ &= m_{1}(\xi) \bigg( \frac{\varphi(\xi+D)-\varphi(\xi)}{\int_{\xi}^{\xi + D} \varphi(s)ds} - (N-1)\sqrt{\sigma}\tanh(\sqrt{\sigma}(\xi+D)\bigg).
\end{align*}
By Cauchy's mean value theorem, there exists $c\in (\xi,\xi+D)$ such that $\frac{\varphi(\xi+D)-\varphi(\xi)}{\int_{\xi}^{\xi + D} \varphi(s)ds} = \frac{\varphi'(c)}{\varphi(c)} = (N-1)\sqrt{\sigma}\tanh(\sqrt{\sigma}c)$. Note that $\tanh x$ is a monotone increasing function, we have $m_{1}(\xi)$ is a monotone increasing function. Hence $g_{\xi}(1) \geq \lim_{\xi \rightarrow \infty} \frac{\int_{\xi}^{\xi + D} \varphi(s)ds}{\varphi(\xi+D)}.$ We have
\begin{align*}
\lim_{\xi \rightarrow \infty} m_{1}(\xi) 
&= \lim_{\xi \rightarrow \infty}  \frac{\varphi(\xi+D)-\varphi(\xi)}{\varphi'(\xi+D)} =   \lim_{\xi \rightarrow \infty}\frac{\varphi(\xi+D)}{\varphi'(\xi+D)} - \lim_{\xi \rightarrow \infty} \frac{\varphi(\xi)}{\varphi(\xi+D)} \frac{\varphi(\xi+D)}{\varphi'(\xi+D)}
\\&= \frac{1}{(N-1)\sqrt{\sigma}}-\frac{1}{(N-1)\sqrt{\sigma}}\lim_{\xi \rightarrow \infty} \frac{\varphi(\xi)}{\varphi(\xi+D)}
\\ &= \frac{1}{(N-1)\sqrt{\sigma}}  - \frac{1}{(N-1)\sqrt{\sigma}}\lim_{\xi \rightarrow \infty} \big( c_{D} + s_{D}\tanh(\sqrt{\sigma}\xi)\big)^{1-N} 
\\& = \frac{1}{(N-1)\sqrt{\sigma}} \big(1- (c_{D}+s_{D})^{1-N}\big).
\end{align*} 
Hence $g_{\xi}(1) \geq \displaystyle\frac{(N-1)(\sqrt{\sigma})}{( 1-(c_{D}+s_{D})^{1-N} )} > 0$.

Now we consider the case where there exists $\theta_{0}\in(0,1)$ such that $c(\theta_{0}) = d_{1,\xi}(\theta_{0})\in (\xi,\xi+D)$. Then 
\[ h_{\xi}(\theta_{0}) = \varphi(c(\theta_{0}))\bigg( \frac{1}{\int_{\xi}^{\xi+D}\varphi(s)ds} - \frac{1}{\int_{-\infty}^{\infty}\varphi(s)ds}\bigg) > 0.\] Note that $\varphi(c(\theta)) \geq \min \{ \varphi(\xi),\varphi(\xi +D)\}$. By a similar calculation as above, we have

\begin{align*} \lim_{\xi \rightarrow -\infty} \varphi(\xi+D)\bigg( \frac{1}{\int_{\xi}^{\xi+D}\varphi(s)ds} - \frac{1}{\int_{-\infty}^{\infty}\varphi(s)ds}\bigg) &= \lim_{\xi \rightarrow -\infty} \varphi(\xi)\bigg( \frac{1}{\int_{\xi}^{\xi+D}\varphi(s)ds} - \frac{1}{\int_{-\infty}^{\infty}\varphi(s)ds}\bigg) \\&= \frac{(1-N)(\sqrt{\sigma})}{( (c_{D}-s_{D})^{N-1} - 1 )} > 0
\end{align*} and
\begin{align*} \lim_{\xi \rightarrow \infty} \varphi(\xi)\bigg( \frac{1}{\int_{\xi}^{\xi+D}\varphi(s)ds} - \frac{1}{\int_{-\infty}^{\infty}\varphi(s)ds}\bigg) &= \lim_{\xi \rightarrow \infty} \varphi(\xi+D)\bigg( \frac{1}{\int_{\xi}^{\xi+D}\varphi(s)ds} - \frac{1}{\int_{-\infty}^{\infty}\varphi(s)ds}\bigg) \\&= \frac{(N-1)(\sqrt{\sigma})}{( 1-(c_{D}+s_{D})^{1-N} )} > 0.
\end{align*}
Taking the infimum over $\xi\in(-\infty,\infty)$ shows 
\[K_{1,D}(\theta) - I_{K,N,\infty}(\theta) \geq \inf_{\xi\in\mathbb{R}} \{ g_{\xi}(0), g_{\xi}(1), h_{\xi}(\theta_{0}) \text{ with } c(\theta_{0}) = d_{1,\xi}(\theta_{0}) \}> 0.\]

\end{proof}

\begin{proposition}
 $I_{K,N,D}(\theta) > I_{K,N,\infty}(\theta)$ for all $\theta \in (0,1)$, $N < 0$ and $D < \infty$.
\end{proposition}
\begin{proof}
This is just a corollary of the lemmas above.
\end{proof}

\end{appendices}

\end{document}